\documentclass[10pt]{amsart}

\usepackage{a4wide}
\usepackage{graphicx} 
\usepackage{geometry}
\usepackage[utf8]{inputenc}
\usepackage{amssymb,amsmath,amscd,amsfonts,amsthm,bbm,mathrsfs,enumerate}
\usepackage{booktabs}
\usepackage[colorlinks=true, linkcolor=blue, citecolor=blue]{hyperref}
\usepackage{mathabx}
\usepackage{graphicx}
\usepackage{color}
\usepackage{accents}
\usepackage{subcaption}
\usepackage{enumitem}
\usepackage{amsaddr}
\usepackage{comment}

\DeclareUnicodeCharacter{00A0}{~}

\newcommand{\bs}{\boldsymbol}

\numberwithin{equation}{section}

\newcommand{\C}{\mathbb{C}}

\newcommand{\Prob}{\mathbb{P}}

\newcommand{\rank}{\mathop{\operator@font rank}}

\newcommand{\vertiii}[1]{{\left\vert\kern-0.25ex\left\vert\kern-0.25ex\left\vert #1
    \right\vert\kern-0.25ex\right\vert\kern-0.25ex\right\vert}}
\makeatother

\newcommand{\N}{\mathbb{N}} 
\newcommand{\E}{\mathbb{E}}

\newtheorem{thm}{Theorem}[section]
\newtheorem{lem}[thm]{Lemma}

\newtheorem{cor}[thm]{Corollary}
\newtheorem{ass}{Assumption}
\theoremstyle{definition}
\newtheorem{defn}[thm]{Definition}

\theoremstyle{remark}
\newtheorem{rem}[thm]{Remark}

\title{A note on outlier eigenvectors for sparse non-Hermitian perturbations}

\author{ Miltiadis Galanis$^{1,3}$, Michail Louvaris$^{2}$}
\address{$^{1}$ Department of Informatics and Telecommunications,
National and Kapodistrian University of Athens, Zografou 161 22, Athens, Greece.
\quad \\ Email: \textit{mgalan@di.uoa.gr} }
\address{$^{2}$ Department of Mathematics, Yale University, New Haven, USA. \\
\quad Email: \textit{michail.louvaris@yale.edu}}
\address{$^{3}$ Quantum Neural Technologies SA, Athens, Greece. \\
\quad Email: \textit{mag@quantumtech.ai}}

\begin{document}
\maketitle

\begin{abstract}
We consider a sparse i.i.d.\ non-Hermitian random matrix model $X_n$ (with sparsity parameter $K_n$)
and a deterministic finite-rank perturbation $E_n$. Assuming biorthogonality for $E_n$ and a growth condition on $K_n$, we outline a finite-rank resolvent
reduction leading to asymptotics for the overlap between an outlier eigenvector of $Y_n:=X_n+E_n$
and the corresponding spike eigenspace. In particular, for an outlier spike $\mu$ with $|\mu|>1$,
the squared projection of the associated (right) eigenvector onto the spike eigenspace converges
in probability to $1-|\mu|^{-2}$. Our result generalizes Theorem 1.6 of \cite{hachem2026extreme} to general finite rank case solving Open Problem 5.
\end{abstract}


\section{Introduction}

The study of eigenvalue outliers in random matrix theory has a long and
well-established history. In the symmetric and Hermitian settings,
additive finite-rank deformations often lead to predictable and
well-understood spectral deviations. A landmark result of Baik, Ben Arous,
and P\'ech\'e (BBP) showed that, for sample covariance matrices with Gaussian
entries, finite-rank deformations produce eigenvalues that detach from the bulk
once a critical threshold is exceeded; see \cite{baik2005phase}. This phase
transition phenomenon was subsequently extended to other models in
\cite{baik2006eigenvalues,paul2007asymptotics,capitaine2009largest,benaych2011eigenvalues}.
In addition to identifying the outlier eigenvalues, these \cite{benaych2011eigenvalues} also
characterized the associated eigenvector overlaps in the Hermitian setting.

In the non-Hermitian i.i.d.\ case, assuming finite fourth moments, the
location of eigenvalue outliers for additive finite-rank deformations was
established in \cite{tao2013outliers,bordenave2016outlier}. However,
a precise description of the associated eigenvectors remained largely open.

Recently in \cite{hachem2026extreme}, the eigenvalues of finite-rank additive perturbations of
\emph{sparse} non-Hermitian random matrices were characterized across all
sparsity regimes and under minimal moment assumptions (see Theorem 1.2  of
\cite{hachem2026extreme}). This was achieved using the convergence framework developed in
\cite{bordenave2022convergence}. Related applications of this framework
appear in \cite{cos-lam-yiz-24,coste2023sparse,hachem2026spectral}.

Under a specific sparsity regime and assuming subgaussian entries, the
asymptotic behavior of the eigenvector projection was determined in the
rank-one case (Theorem 1.6 of \cite{hachem2026extreme}), using universality results from
\cite{brailovskaya2024universality}. In that setting, the squared overlap
between the outlier eigenvector and the spike direction converges to
$1-|\mu|^{-2}$ for spikes $\mu$ outside the unit disk.

\medskip

The purpose of the present note is to remove the rank-one restriction and
to establish the corresponding eigenvector behavior for general finite-rank
deterministic perturbations. More precisely, we consider sparse non-Hermitian
random matrices $X_n$ and deterministic perturbations $E_n$ of arbitrary fixed
rank. We quantify the alignment between an outlier eigenvector of
\[
Y_n = X_n + E_n
\]
and the corresponding spike eigenspace of $E_n$. We make the assumption that $E_n$ admits a biorthogonal representation.

The extension from rank one to finite rank is not merely notational.
In the non-Hermitian setting, multiplicities and interactions between
distinct spike blocks introduce genuine structural difficulties.
In particular, one must control the kernel of a finite-dimensional
matrix-valued function derived from the resolvent and localize the
associated kernel vector onto the correct spike block.
The argument therefore requires a systematic finite-rank resolvent
reduction and a quantitative kernel localization mechanism.

\medskip

Our approach is entirely resolvent-based. We first establish a
finite-rank kernel--eigenspace bijection, which expresses any outlier
eigenvector in terms of the resolvent of $X_n$ and a low-dimensional
kernel vector. The main task is then to show that this kernel vector
concentrates on the appropriate spike block and that the compressed
resolvents converge to their deterministic limits. Combining these
ingredients yields the asymptotic overlap formula
\[
\langle \tilde u_{\ell,n}, F_{\ell,n} \rangle^2
\xrightarrow{\mathbb{P}}
1 - \frac{1}{|\mu|^2},
\]
for spikes $\mu$ with $|\mu|>1$. Notice that the limit is the same as in the Hermitian case; see \cite{benaych2011eigenvalues}.

\medskip

Additively deformed non-Hermitian matrices arise naturally in several
applied fields. In neural network theory, the matrix $Y_n$ models random
interactions between neurons \cite{som-cri-som-88,wai-tou-13}.
In theoretical ecology, sparse interaction matrices describe the
dynamics of ecosystems \cite{bunin2017ecological,akj-etal-24}.
Understanding the stability and structure of outlier modes is therefore
relevant in these contexts.

\medskip

The paper is organized in such a way that the linear-algebraic reduction is explicit
and reusable. After establishing the finite-rank reduction, we combine
resolvent estimates and universality results to control the relevant
bilinear forms and complete the proof of the main theorem.

\section{Results}
\subsection{Notation}
Throughout, $\langle\cdot,\cdot\rangle$ denotes the standard Hermitian inner product on $\C^n$ and
$\|\cdot\|$ the matrix operator norm or the vector Euclidean norm. Moreover denote by
$\sigma(M)=\{\lambda_1(M), \ldots, \lambda_m(M)\}$ the spectrum of an $m\times m$ matrix $M$.
Furthermore for a sequence of random variables $J_n$ and a random variable $J$ we write
\[
J_n\xrightarrow{\mathbb{P}}J
\]
to denote convergence in probability, and we write $J_n=o_{\mathbb{P}}(1)$ to denote
$J_n\xrightarrow{\mathbb{P}}0$.  For
$m\in \mathbb{N}$, set $[m] = \emptyset$ if $m = 0$ and $[m] = \{1,\ldots, m
\}$ otherwise. 

Lastly recall the definition of Hausdorff distance between two sets. Let $z\in\C$ and
$A,B\subset\C$. Define $d(z,A)=\inf_{\xi\in A}|z-\xi|$. The Hausdorff distance between $A$ and $B$,
denoted by $d_{\bs H}(A,B)$, is
\[
d_{\bs H}(A,B)=\max\left\{\sup_{z\in A} d(z,B)\,;\ \sup_{z\in B} d(z,A)\right\}.
\]

\subsection{Model}
Let $\chi$ be a complex-valued random variable such that $\E(\chi)=0$ and $\E(|\chi|^2)=1$.
For each integer $n\ge 1$, let $A_n=(\{A_n\}_{ij})_{i,j=1}^n\in\C^{n\times n}$ be a random matrix with
i.i.d.\ entries distributed as $\chi$.
Let $(K_n)$ be a sequence of positive integers with $K_n\le n$. Let $(B_n)$ be a sequence of
$n\times n$ matrices with i.i.d.\ Bernoulli entries such that
\[
\Prob\{\{B_{n}\}_{1,1}=1\}=K_n/n,
\]
and assume $B_n$ and $A_n$ are independent. Define $X_n=(\{X_n\}_{ij})_{i,j=1}^n$ by
\begin{equation}\label{sparser_matrices}
\{X_n\}_{ij}=\frac{1}{\sqrt{K_n}}\,\{B_n\}_{ij}\,\{A_n\}_{ij}.
\end{equation}
Then $\E \{X_n\}_{11}=0$ and $\E|\{X_n\}_{11}|^2=1/n$. 
The parameter $K_n$ is referred to as the sparsity parameter of $X_n$.

Let $r>0$ be fixed. Consider deterministic vectors $u^{1,n},\dots,u^{r,n},v^{1,n},\dots,v^{r,n}\in\C^n$
and define the deterministic finite-rank perturbation
\[
E_n=\sum_{t=1}^r u^{t,n}(v^{t,n})^\star.
\]
Define $$Y_n:=X_n+E_n.$$

We make the following assumption for $E_n.$
\begin{ass}\label{ass:E-bounded}
There exists an absolute constant $C>0$ such that
\[
\sum_{t=1}^r \|u^{t,n}\| + \|v^{t,n}\|\le C;
\]
\end{ass}

Recently the following result was proven in \cite{hachem2026extreme} for the eigenvalues of $Y_n$. 
\begin{thm}\label{thm_eigenvalues}[Theorem 1.2 of \cite{hachem2026extreme}]
Assume that 
\[
K_n \xrightarrow{n\to\infty} \infty
\]
and that Assumption \ref{ass:E-bounded} holds true.
Define
\[
\sigma^+(E_n)=\sigma(E_n)\cap\{z\in\C:\ |z|>1\},
\qquad
\sigma^+_\varepsilon(Y_n)=\sigma(Y_n)\cap\{z\in\C:\ |z|\ge 1+\epsilon\},
\]
and let $m_n=|\sigma^+(E_n)|$. Then
\[
\Prob\Big\{|\sigma^+_\varepsilon(Y_n)|\neq m_n\Big\}\xrightarrow[n\to\infty]{}0.
\]
For each sequence $(n')$ with $n'\to\infty$ and $m_{n'}>0$ for all $n'$,
\[
d_{\bs H}\big(\sigma^+_\varepsilon(Y^{n'}),\sigma^+(E^{n'})\big)\xrightarrow[n\to\infty]{\Prob} 0,
\]
with the convention $d_{\bs H}(\emptyset,\sigma^+(E^{n'}))=\infty$.
\end{thm}

\begin{ass}\label{ass:K}
The sequence $(K_n)$ satisfies
\[
\frac{\log^9 n}{K_n}\xrightarrow[n\to\infty]{}0.
\]
and that 
there exists an absolute constant $C > 0$ such that 
\[
 \mathbb{P}(|A_{11}| \geq t) \ \leq\  2 \exp(-C t^2)\,. 
\] 
that is $\chi$ follows a sub-gaussian law.
\end{ass}

Moreover we have the following result concerning the eigenvectors of rank $1$ perturbation.
 
\begin{thm}\label{thm_eigenvector_r=1}[Theorem 1.6 of \cite{hachem2026extreme}]
Assume $r=1$ so $E_n=u_n(v_n)^\star $ and $Y_n=X_n+u_n(v_n)^\star$. Moreover assume that 
\[
\liminf_{n \to \infty} \left| \langle v_n,u_n\rangle \right|  \ >\  1\, 
\] 
and that Assumption \ref{ass:K} holds true.
 Recall the notation from Theorem \ref{thm_eigenvalues}.
When the event $\left\{ | \sigma^+_\varepsilon(Y_n) | = 1 \right\}$ is realized, 
let $\tilde{u}_n$ be an 
unit-norm right eigenvector of $Y_n$ corresponding to $\lambda_{\max}(Y_n)$.  
Otherwise, put $\tilde{u}_n = 0_n$. Then, it holds that 
\[
 \left|\left\langle\tilde{u}_n,\frac{u_n}{\|u_n\|} \right\rangle\right|^2 - \left( 1-\frac 1{|\langle u_n,v_n\rangle|^2}\right)\quad \xrightarrow[n\to\infty]{\mathbb{P}}\quad 0\, .
\]
\end{thm}

\begin{rem}
Assumption \ref{ass:K} is needed in order to give an upper bound for $\|(X_n-\langle v_n, u_n \rangle I)^{-1}\|$, which is a necessary tool in order to compute and compare the outlier eigenvectors, see for example Corollary \ref{cor:Mn-converges-outside}. This is achieved by using the universality results from \cite{brailovskaya2024universality}. We shall make use of these results in this paper as well, see Section \ref{section_resolvent_bound}.  
\end{rem}

Our main goal will be to generalize Theorem \ref{thm_eigenvector_r=1} to a general rank $r\geq 1.$

In order to achieve that we will need some assumptions on $E_n$. 
\begin{ass}\label{ass:E}
There exist $\delta>0$ and distinct complex numbers $\mu^{(1)},\dots,\mu^{(m)}$ with $|\mu^{(\ell)}|\ge 1+\delta$
such that, for all $n$ large enough,
\begin{itemize}
\item[(i)] For $n$ large enough, $E_n$ admits a biorthogonal decomposition
\[
E_n=P_n\Lambda_n W_n^*,\qquad W_n^*P_n=I_r,
\]
with $\Lambda_n$ diagonal with entries the eigenvalues of $E_n$. We also assume that $P_n$ and $W_n$ have rank $r$ for large enough $n.$
\item[(ii)]It is true that
$\sigma_n^+(E_n):=\sigma(E_n)\cap\{z:|z|>1\}=\{\mu^{(1,n)},\dots,\mu^{(m_n,n)}\}$ (counting geometric multiplicity). Then 
\[
    d_{\bs H}\left(\sigma_n(E^n), \left\{\mu^{(1)},\dots,\mu^{(m)}\right\}\right) \xrightarrow[]{n\to \infty}0.
\]
Moreover for each $\ell\in\{1,\dots,m_n\}$, we assume that the (right) spike eigenspace
\[
F_{\ell,n}:=\ker(\mu^{(\ell,n)}I-E_n)\subset\C^n
\]
satisfies
\[
\qquad \lim_{n\to \infty}k_{\ell,n}= k_{\ell}\le r.
\]
where $\dim F_{\ell,n}=k_{\ell,n}$.
\end{itemize}
\end{ass}

\begin{rem}
 In Assumption \ref{ass:E} (ii) we assume the set of eigenvalues of $E_n$ converge to the set $\{ \mu^{(1)},\cdots \mu^{(m)}$\} . This is done mainly for expositional reasons. One may avoid this Assumption and state our main result, Theorem \ref{main_thm}, as Theorem \ref{thm_eigenvector_r=1} is stated. Moreover Assumption \ref{ass:E} (i) makes our computations cleaner, see Lemma \ref{lem:ker-evec-bijection} and \eqref{M_n()+1/mu L_n} for example. We believe that one can avoid this Assumption and restate the result in terms of the Jordan blocks of $E_n$. We do not pursue this direction.
\end{rem}
Next we present some notation and definitions.

Let $Q_{\ell,n}\in\C^{n\times k_{\ell,n}}$ have orthonormal columns spanning $F_{\ell,n}$ and set
\begin{align}\label{defn_P,Q}
P_{\ell,n}:=Q_{\ell,n}Q_{\ell,n}^*=\operatorname{Proj}_{F_{\ell,n}}.
\end{align}
Moreover we have the following definition.
\begin{defn}
Let $F\subset\C^n$ be a deterministic linear subspace and $x\in\C^n$.
We denote by $\langle x,F\rangle$ the norm of the orthogonal projection of $x$ onto $F$, i.e.
\[
\langle x,F\rangle:=\|\operatorname{Proj}_F x\|,\qquad \langle x,F\rangle^2=\|\operatorname{Proj}_F x\|^2.
\]
Equivalently, if $Q\in\C^{n\times k}$ has orthonormal columns spanning $F$ (so $Q^*Q=I_k$ and $\mathrm{Ran}(Q)=F$),
then $\operatorname{Proj}_F=QQ^*$ and
\begin{equation}\label{eq:BGN-inner-product}
\langle x,F\rangle^2=\|Q^*x\|^2=\sum_{j=1}^k |\langle x,q_j\rangle|^2.
\end{equation}
In particular, if $F=\mathrm{span}\{u\}$ is one-dimensional, then
\begin{equation}\label{eq:BGN-1d}
\langle x,F\rangle^2=\frac{|\langle x,u\rangle|^2}{\|u\|^2}.
\end{equation}
\end{defn}

\begin{thm}\label{main_thm}
Fix $\ell\in\{1,\dots,m\}$ and write $\mu:=\mu^{(\ell)}$. Let Assumptions \ref{ass:K} and \ref{ass:E} hold true. By Theorem \ref{thm_eigenvalues} there is some $\lambda_{\ell,n} \in \sigma(Y_n)$ such that 
\[
\lambda_{\ell,n}    \xrightarrow[n\to\infty]{\Prob} \mu.
\]
Moreover by Assumption \ref{ass:E} there is some sequence $\mu_n \in \sigma(E_n)$ such that 
\[
    \mu_n \to \mu.
\]
Set $F_n:=\operatorname{ker}(\mu_nI-E_n)$
and let $\tilde u_{\ell,n}$ denote a unit right eigenvector associated with $\lambda_{\ell,n}$. Then
\begin{enumerate}
\item 
\begin{equation}\label{eq:BGN-goal}
\langle \tilde u_{\ell,n},F_n\rangle^2
=\|Q_{\ell,n}^*\tilde u_{\ell,n}\|^2
\ \xrightarrow{\Prob}\ 1-\frac{1}{|\mu|^2}.
\end{equation}

\item  For any sequence $\mu'_n \in \sigma(Y_n)$ such that 
\[
    \mu'_n \to \mu' \neq \mu
\]
if one sets $F_{\ell',n}=\operatorname{ker}(\mu'_n I-E_n)$ and assumes that $F_n \perp F_{\ell',n}$ for all $n$ large enough
it is true that
\[
\langle \tilde u_{\ell,n}, F_{\ell',n}\rangle^2 \ \xrightarrow{\mathbb{P}} 
0.
\]
\end{enumerate}
\end{thm}
\begin{rem}
    For Theorem \ref{main_thm}(b) the assumption that $F_n \perp F_{\ell',n}$ clearly is true when $E_n$ is diagonalizable. If one omits this assumption our result states that for some sequence $c_n \in \C^{k_{l,n}}$  of unit vectors, 
    \begin{align*}
      \langle \tilde u_{\ell,n}, F_{\ell',n}\rangle^2 - \frac{|\mu|^2-1}{|\mu|}\,\|Q_{\ell',n}^*Q_{\ell,n}c_n\|^2 \ \xrightarrow{\mathbb{P}} 
0.
    \end{align*}
    Here $Q_{\ell',n}$ and $Q_{\ell,n}$ are as in \eqref{defn_P,Q}.
\end{rem}
\section{Tools from Linear Algebra}
We start with some results from linear algebra that provide a convenient expression for the projection in Theorem \ref{main_thm} in terms of quantities we can control.

\begin{lem}[Finite-rank reduction: kernel--eigenspace bijection]\label{lem:ker-evec-bijection}
Let $X\in\C^{n\times n}$ and let $U,V\in\C^{n\times r}$ with $\operatorname{rank}(U)=r$ (equivalently, $U$ has full column rank).
Set $Y:=X+UV^*$.
Fix $\lambda\in\C$ such that $\lambda\notin\sigma(X)$, and define
\[
R(\lambda):=(X-\lambda I)^{-1},\qquad M(\lambda):=V^*R(\lambda)U\in\C^{r\times r}.
\]
Define $\Phi:\ker(I_r+M(\lambda))\to \ker(Y-\lambda I)$ by $\Phi(a):=R(\lambda)Ua$.
Then $\Phi$ is a linear bijection. In fact, for every $x\in\ker(Y-\lambda I)$ one has
\[
\Phi^{-1}(x)=-\,V^*x,
\]
and consequently $\dim\ker(I_r+M(\lambda))=\dim\ker(Y-\lambda I)$.
\end{lem}

\begin{proof}
\textbf{Step 1: $\Phi$ is well-defined.}
Let $a\in\ker(I_r+M(\lambda))$. Using $Y-\lambda I=(X-\lambda I)+UV^*$ and $(X-\lambda I)R(\lambda)=I$,
\begin{align*}
(Y-\lambda I)\Phi(a)
&=\bigl((X-\lambda I)+UV^*\bigr)R(\lambda)Ua \\
&=(X-\lambda I)R(\lambda)Ua+UV^*R(\lambda)Ua \\
&=Ua+U\bigl(V^*R(\lambda)U\bigr)a \\
&=U(I_r+M(\lambda))a \\
&=0.
\end{align*}

\textbf{Step 2: $\Phi$ is injective.}
If $\Phi(a)=0$, then $R(\lambda)Ua=0$, hence $Ua=0$. Since $U$ has full column rank, $a=0$.

\textbf{Step 3: $\Phi$ is surjective (and compute $\Phi^{-1}$).}
Let $x\in\ker(Y-\lambda I)$. Then $(X-\lambda I)x+UV^*x=0$, so
\[
x=-R(\lambda)Ua,\qquad a:=V^*x.
\]
Applying $V^*$ yields $a=-M(\lambda)a$, hence $(I_r+M(\lambda))a=0$. Thus $x=\Phi(-a)$.
Moreover $\Phi^{-1}(x)=-V^*x$.
\end{proof}

As a result of the previous lemma we have the following corollary.

\begin{cor}[Closed-form representation of the unit outlier eigenvector]\label{cor:unit-evec-rep}
Let $X\in\C^{n\times n}$ and let $E=UV^*$ with $U,V\in\C^{n\times r}$ and $\operatorname{rank}(U)=r$.
Set $Y:=X+UV^*$. Fix $\lambda\in\C$ with $\lambda\notin\sigma(X)$ and define
\[
R(\lambda):=(X-\lambda I)^{-1},\qquad M(\lambda):=V^*R(\lambda)U.
\]
Assume $\lambda\in\sigma(Y)\setminus\sigma(X)$ and let $\tilde u$ be any unit right eigenvector of $Y$
associated with $\lambda$. Then there exists $a\in\ker(I_r+M(\lambda))\setminus\{0\}$ such that
\begin{equation}\label{eq:unit-evec-form}
\tilde u=\frac{R(\lambda)\,Ua}{\|R(\lambda)\,Ua\|}.
\end{equation}
Moreover, $a$ is unique up to multiplication by a nonzero scalar.
\end{cor}

\begin{proof}
By Lemma~\ref{lem:ker-evec-bijection}, $\Phi(a)=R(\lambda)Ua$ is a bijection from $\ker(I_r+M(\lambda))$ to $\ker(Y-\lambda I)$.
Since $\tilde u\in\ker(Y-\lambda I)$ and $\tilde u\neq 0$, there exists $a\neq 0$ with $\tilde u=\Phi(a)/\|\Phi(a)\|$.
Uniqueness up to scaling follows from injectivity of $\Phi$.
\end{proof}

Thus we are interested in \eqref{eq:BGN-goal} for $\tilde u_{\ell,n}$ as in \eqref{eq:unit-evec-form}.
Next we give an approximation for the projections.

\begin{lem}\label{lem:kernel-localization-from-approx}
Recall Assumption \ref{ass:E} for $E_n$. Fix $\mu\in\C$ and assume that $\Lambda_n$ contains a block $\mu I_k$ of size $k\ge 1$, i.e.
\[
\Lambda_n=\operatorname{diag}(\mu I_k,\Lambda_{\neq,n}),
\]
with $\Lambda_{\neq,n}\in\C^{(r-k)\times(r-k)}$ diagonal.

Define the spike-adapted rank factorization
\begin{equation}\label{eq:UV-spike-adapted}
U_n:=P_n,\qquad V_n:=W_n\overline{\Lambda_n},
\end{equation}
so that $E_n=U_nV_n^*$ and $V_n^*U_n=\Lambda_n$.

Notice that due to Assumption \ref{ass:E} there is $c_0>0$ such that
\begin{equation}\label{eq:sep}
\min_{\nu\in\sigma(\Lambda_{\neq,n})}\Big|1-\frac{\nu}{\mu}\Big|\ \ge\ c_0.
\end{equation}

For any $C_n\in\C^{n\times n}$ and $z\notin\sigma(C_n)$ define
\[
J_n(z):=(C_n-zI)^{-1},\qquad N_n(z):=V_n^*J_n(z)U_n\in\C^{r\times r}.
\]
Let $\lambda_n\in\C$ satisfy $\lambda_n\notin\sigma(J_n)$, and assume that for some $\varepsilon_n\in(0,c_0/2)$,
\begin{equation}\label{eq:Mn-approx}
\Big\|\big(I_r+N_n(\lambda_n)\big)-\Big(I_r-\frac{1}{\mu}\Lambda_n\Big)\Big\|\ \le\ \varepsilon_n.
\end{equation}

Let $a_n\in\ker(I_r+N_n(\lambda_n))\setminus\{0\}$ and decompose $a_n=(a_{\mu,n},a_{\neq,n})$ according to
$\C^r=\C^k\oplus\C^{r-k}$. Then:
\begin{enumerate}
\item[(i)] $a_{\mu,n}\neq 0$.
\item[(ii)] The off-resonant component is small:
\begin{equation}\label{eq:a-ne-small}
\|a_{\neq,n}\|\ \le\ \frac{\varepsilon_n}{c_0-\varepsilon_n}\,\|a_{\mu,n}\|
\ \le\ \frac{2}{c_0}\,\varepsilon_n\,\|a_{\mu,n}\|.
\end{equation}
\end{enumerate}
\end{lem}

\begin{proof}

Set
\[
K_n:=I_r+N_n(\lambda_n),\qquad D_n:=I_r-\frac{1}{\mu}\Lambda_n.
\]
Then
\[
D_n=\begin{pmatrix}
0_{k\times k} & 0\\
0 & D_{22,n}
\end{pmatrix},
\qquad
D_{22,n}:=I_{r-k}-\frac{1}{\mu}\Lambda_{\neq,n}.
\]
By \eqref{eq:sep}, $D_{22,n}$ is invertible and
\begin{equation}\label{eq:D22-inv}
\|D_{22,n}^{-1}\|\le \frac{1}{c_0}.
\end{equation}

Write
\[
K_n=\begin{pmatrix}
K_{11,n} & K_{12,n}\\
K_{21,n} & K_{22,n}
\end{pmatrix}.
\]
The bound \eqref{eq:Mn-approx} implies
\begin{equation}\label{eq:block-bounds}
\|K_{21,n}\|\le \varepsilon_n,\qquad \|K_{22,n}-D_{22,n}\|\le \varepsilon_n.
\end{equation}

\textbf{Step 1: $K_{22,n}$ is invertible.}
Let $E_{22,n}:=K_{22,n}-D_{22,n}$. Then $\|D_{22,n}^{-1}E_{22,n}\|\le \varepsilon_n/c_0<1/2$.
Hence $K_{22,n}=D_{22,n}(I+D_{22,n}^{-1}E_{22,n})$ is invertible and
\begin{equation}\label{eq:K22-inv}
\|K_{22,n}^{-1}\|\le \frac{1}{c_0-\varepsilon_n}.
\end{equation}

\textbf{Step 2: kernel localization.}
Let $a_n=\binom{a_{\mu,n}}{a_{\neq,n}}\in\ker(K_n)\setminus\{0\}$. From the second block row,
\[
K_{21,n}a_{\mu,n}+K_{22,n}a_{\neq,n}=0,
\]
so
\[
a_{\neq,n}=-K_{22,n}^{-1}K_{21,n}a_{\mu,n}.
\]
Taking norms and using \eqref{eq:block-bounds} and \eqref{eq:K22-inv} gives \eqref{eq:a-ne-small}.
If $a_{\mu,n}=0$ then $a_{\neq,n}=0$, contradicting $a_n\neq 0$.
\end{proof}

\section{Results on bilinear forms of the resolvent of $X_n$.}\label{section_resolvent_bound}
In what follows for any $z $ not an eigenvalue of $X_n$ set $R_n(z)=(X_n-zI)^{-1}$.

\begin{lem}\label{isotropic}
Let $u_n$ and $v_n$ be two sequences of vectors in $\C^n$ such that there is some $C>0$ for which $\|u_n\|,\|v_n\|<C$ for all $n \in \N$.
Then for any $|z|>1$ let $\mathcal{E}_{\text{inv},n}(z)$ denote the event that $(X_n-z)$ is invertible. Then
\begin{equation}\label{R_ninvert}
    \mathbb{P}\big( \mathcal{E}_{\text{inv},n}(z)\big) \to 1.
\end{equation}
Moreover we have the following approximation
\[
1_{\mathcal{E}_{\text{inv},n}(z)}\left(
\langle R_n(z)u_n, v_n\rangle + \frac{1}{z}\langle u_n, v_n \rangle \right)\xrightarrow{\Prob} 0.
\]
\end{lem}

\begin{proof}
The first part of the lemma, \eqref{R_ninvert}, follows from Lemma 4.2 of \cite{hachem2026extreme}.

For the second part we shall assume without generality loss that 
\[
 \langle u_n,v_n\rangle \ \xrightarrow[n\to\infty]{} \ \xi \in \C , 
\]
since it is sufficient to establish convergence in probability along all subsequential limits of
$\langle u_n,v_n\rangle$.

We first prove the claim when $\xi \neq 0$. In this case one may set 
\[
\tilde u_n= \frac{u_n}{|\xi|^{1/2} (1-\epsilon)} \ \  \text{ and }\ \ \ \tilde{v}_n= \frac{v_n}{|\xi|^{1/2} (1-\epsilon)}
\]
for $\epsilon>0$ small enough. Then for all $n$ large enough
\[
   |\langle \tilde u_n, \tilde v_n \rangle| >1.
\]
Clearly it is sufficient to prove
\[
1_{\mathcal{E}_{\text{inv},n}(z)}\left(
\langle R_n(z)\tilde u_n, \tilde v_n\rangle + \frac{1}{z}\langle \tilde u_n, \tilde v_n \rangle \right)\xrightarrow{\Prob} 0.
\]
The latter can be proven exactly as Lemma 4.3 of \cite{hachem2026extreme}.

It remains to prove the claim in the case where $\xi =0$. Then we may assume $u_n\neq 0$ and $v_n\neq 0$ for all $n$ large enough, else the claim follows trivially.

For $\epsilon>0$ we set $\bar v_n=\frac{\epsilon}{\|u_n\|^2} u_n+v_n$. Then 
\[
    \langle   u_n, \bar v_n\rangle= \frac{\epsilon}{\| u_n\|^2} \langle u_n, u_n\rangle+ \langle u_n,v_n\rangle\xrightarrow{n \to \infty }\epsilon \neq 0.
\]
But we already have proven that
\[
1_{\mathcal{E}_{\text{inv},n}(z)}\left(
\langle R_n(z) u_n, v_n\rangle + \frac{\epsilon}{\|u_n\|^2}\langle  R_n(z) u_n, u_n \rangle + \frac{\epsilon}{z}\right)\xrightarrow{\Prob} 0.
\]
The claim now follows since 
\[
1_{\mathcal{E}_{\text{inv},n}(z)}\left( \frac{\epsilon}{\|u_n\|^2}\langle  R_n(z) u_n, u_n \rangle + \frac{\epsilon}{z}\right)\xrightarrow{\Prob} 0.
\]
\end{proof}

\begin{cor}\label{cor:Mn-converges-outside}
Let $C^n_1$ and $C^n_2$ be two sequences of $k_1\times n$ and $k_2 \times n$ matrices for some $k_1,k_2 \in \N$. Assume that there is some $C>0$ 
\begin{equation}\label{boundingC_1,2}
    \| C^n_1\|, \|C^n_2 \|< C, \ \ \ \text{ for all } n.
\end{equation}
Recall the event $\mathcal{E}_{\text{inv},n}(z)$ from Lemma \ref{isotropic}. Then  
\[
1_{\mathcal{E}_{\text{inv},n}(z)}\left\| C^n_1 R_n(z) (C^n_2)^*- \frac{1}{z} C^n_1 (C^n_2)^* \right\| \xrightarrow{\Prob} 0.
\]
\end{cor}

\begin{proof}
On the event $\mathcal{E}_{\text{inv},n}(z)$ notice that the $i,j$-th entry of $C^n_1 R_n(z) (C^n_2)^*$ is equal to 
\[
    (C^n_1)_i\, R_n(z)\, (C_2^n)^*_j,
\]
where $(C^n_1)_i$ (respectively $(C_2^n)^*_j$) denotes the $i$-th row of $C^n_1$ (respectively the $j$-th column of $(C_2^n)^*$). Thus due to \eqref{boundingC_1,2}, we may apply Lemma \ref{isotropic} entrywise and use the fact that for any $k_1\times k_2$ matrix $J$,
\begin{align}\label{J_ij}
    \| J\|\leq k_1 k_2 \max_{i,j \in [k_1]\times [k_2]} |J_{i,j}|
\end{align}  to conclude. In \eqref{J_ij}, $J_{i,j}$ denotes the $(i,j)$-th entry of $J$.
\end{proof}

\begin{lem}\label{lemma_R_n_z_n}
Let $(z_n)_{n\ge 1}$ satisfy $|z_n|\ge 1+\varepsilon$ for all $n$ and $z_n\to z$ with $|z|\ge 1+\varepsilon$.
There is some absolute constant $c>0$ such that if one sets $\mathcal{E}_{\text{bound},n}(z)$ to be the event where $(X_n-z)$ and $(X_n-z_n)$ are invertible and 
\[
    \| R_n(z)\|,\| R_n(z_n)\|<c,
\]
then it is true that 
\[
   \mathbb{P}(\mathcal{E}_{\text{bound},n}(z)) \xrightarrow[n \to \infty]{}  1.
\]
Moreover on this event
\begin{equation}\label{R_n(z_n)-R_n(z)}
1_{\mathcal{E}_{\text{bound},n}(z)}  \|R_n(z_n)-R_n(z)\| \  \xrightarrow[n \to \infty]{}\ 0.
\end{equation}
\end{lem}

\begin{proof}
The first part of the lemma follows from Lemma 4.2 of \cite{hachem2026extreme}.
The second part follows from the resolvent identity,
\[
R_n(z_n)-R_n(z)=(z-z_n)R_n(z_n)R_n(z),
\]
and the first part of the lemma.
\end{proof}

We continue with the following Lemma.

\begin{lem}\label{lemnormofR}
Fix $|z|>1$ and recall the event $\mathcal{E}_{\text{bound},n}(z)$ from Lemma \ref{lemma_R_n_z_n}. Then for any sequence of deterministic sequence of vectors $w_n \in \C^n$ such that $\| w_n \| < C$  for some $C>0$ and for all $n$, it is true that
\[
  1_{\mathcal{E}_{\text{bound},n}(z)}\left|\|R_n(z)w_n\|^2 - \frac{\|w_n\|^2}{\sqrt{|z|^2-1}} \right|\ \ \xrightarrow{\Prob}  0 .
\]
\end{lem}
\begin{proof}
    We may assume that $\| w^n\|>0$ for all $n$ , else the claim follows trivially. In this case the claim follows by Lemma 4.6 of \cite{hachem2026extreme}.
\end{proof}
We conclude this section with the following Corollary.
\begin{cor}\label{cor_bound_matrix_product_norm}
Fix $z\in \C: |z|>1$ and recall the event $\mathcal{E}_{\text{bound},n}(z)$ from Lemma \ref{R_n(z_n)-R_n(z)}.  Let $C^n_1$ and $C^n_2$ be two sequences of $k_1\times n$ and $k_2 \times n$ matrices respectively for some $k_1,k_2 \in \N$ such that
\begin{equation}\label{boundingC_1}
    \| C^n_1\|,\| C_2^n\| < C, \ \ \ \text{ for all } n.
\end{equation}
 for some $C>0$. Then 
 \begin{align}
    1_{\mathcal{E}_{\text{bound},n}(z)} \left\| (C^n_1)^* R^*_n(z) R_n(z) C^n_1 - \frac{1}{|z|^2-1}  (C^n_1)^* C^n_2 \right\| \xrightarrow{\Prob}  0 .
 \end{align}
\end{cor}
\begin{proof}
    We start by noticing that by the polarization identity for any sequence of deterministic vectors $x_n, y_n\in \C^n$
    \begin{align}
      & 4 \langle x_n, R^*_n(z) R_n(z) y_n \rangle= 4 \langle R_n(z) x_n, R_n(z) y_n\rangle= \\& = \| R_n(z)(x_n+y_n)\|^2-\|R_n(z) (x_n-y_n) \|^2 +i \|R_n(z)(x_n+iy+n)\|^2-i  \|R_n(z)(x_n-iy_n)\|^2
    \end{align}
    In particular by Lemma \ref{lemnormofR} if $\|x_n\|,\|y_n\|<C$ for all $n$ then
\begin{align}\label{innerR_nx,y}
        \langle R_n(z) x_n, R_n(z) y_n\rangle = \frac{1}{|z|^2-1} \langle x_n, y_n \rangle+o_{\mathbb{P}}(1).
    \end{align}
    It remains to apply \eqref{innerR_nx,y} to each entry and bound the operator norm as in Corollary \ref{cor:Mn-converges-outside}.
\end{proof}
\section{Proof of Theorem \ref{main_thm}}
We start with the proof of Theorem \ref{main_thm}(a).

\begin{proof}[Proof of Theorem \ref{main_thm}(a)]
In what follows set $M_n(z):=V_n^*R_n(z)U_n$ for any $z$ which isn't an eigenvalue of $X_n$. Here $V_n$ and $U_n$ are as in Lemma \ref{lem:kernel-localization-from-approx}.  Moreover, without loss of generality, we will assume that the first $k_{\ell,n}$ diagonal entries of $\Lambda_n$ are equal to $\mu_n$ and that one has the decomposition $U_n=[U_{\mu_n,n},U_{\neq,n}]$ where $U_{\mu_n,n} \in \C^{n \times k_{\ell,n}}$ corresponds to the right eigenvectors of $E_n$ with eigenvalue $\mu_n$.

Denote by $\mathcal{E}_{\text{inv},n}(\lambda_{\ell,n})$ the event that $(X_n-\lambda_{\ell,n})$ is invertible.
By Corollary \ref{cor:unit-evec-rep} there is some non-zero vector $a_n\in\ker\!\big(I_r+M_n(\lambda_{\ell,n})\big)$ such that
\begin{equation}\label{1st_equation_projection}
1_{\mathcal{E}_{\text{inv},n}(\lambda_{\ell,n})}\|Q_{\ell,n}^*\tilde u_{\ell,n}\|^2
=1_{\mathcal{E}_{\text{inv},n}(\lambda_{\ell,n})}
\frac{\|Q_{\ell,n}^*R_n(\lambda_{\ell,n})U_na_n\|^2}{\|R_n(\lambda_{\ell,n})U_na_n\|^2}.
\end{equation}
We can also assume that $\|U_n a_n\|=1,$ since $U_n$ has rank $r$ and so $U_n a_n \neq 0$ and $a_n$ can be rescaled arbitrarily.

Moreover due to Assumptions \ref{ass:E} we may apply Lemma \ref{lem:kernel-localization-from-approx}
to conclude that we may decompose $a_n=(a_{\mu,n},a_{\neq,n})$ so that
\begin{equation}\label{a_neq_small}
\|a_{\neq,n}\|\leq C \Big\|M_n(\lambda_{\ell,n})+\frac{1}{\mu}\Lambda_n\Big\|\ \  \text{ on }\mathcal{E}_{\text{inv},n}( \lambda_{\ell,n})
\end{equation}
for some absolute constant $C>0$. Furthermore recall the event $\mathcal{E}_{\text{bound},n}(\mu)$ from Lemma \ref{lemma_R_n_z_n}. On this event, which is a subset of  $\mathcal{E}_{\text{inv},n}(\lambda_{\ell,n})$ it holds true that
\[
    1_{\mathcal{E}_{\text{bound},n}(\mu)} \Big\|M_n(\lambda_{\ell,n})+\frac{1}{\mu}\Lambda_n\Big\| \leq  1_{\mathcal{E}_{\text{bound},n}(\mu)} \| V_n\| \| U_n \|\, \| R_n(\lambda_{\ell,n})-R_n(\mu)\|+   1_{\mathcal{E}_{\text{bound},n}(\mu)}  \Big\|M_n(\mu)+\frac{1}{\mu}\Lambda_n\Big\|.
\]

In particular by Lemma \ref{lemma_R_n_z_n} and Corollary \ref{cor:Mn-converges-outside} we get that
\begin{equation}\label{M_n()+1/mu L_n}
1_{\mathcal{E}_{\text{bound},n}(\mu)}  \Big\|M_n(\lambda_{\ell,n})+\frac{1}{\mu}\Lambda_n\Big\|\xrightarrow{\Prob}0.
\end{equation}

Thus using \eqref{a_neq_small} and \eqref{M_n()+1/mu L_n} we conclude that if one sets the $r$-dimensional vector $\tilde{a}_{\mu,n}=(a_{\mu,n},0)$, i.e. the first $k_{\ell,n}$ coordinates of $\tilde{a}_{\mu,n}$ are equal to the coordinates of $a_{\mu,n}$ and the rest are equal to $0$, then 
\[
1_{\mathcal{E}_{\text{bound},n}(\mu)}  \|a_n-\tilde{a}_{\mu,n}\|=o_{\Prob}(1).
\]
In particular since we have assumed that $U_na_n$ is a unit vector
\[
     \|U_n \tilde{a}_{\mu,n} \|=  \|U_n a_n\|+o_{\mathbb{P}}(1). \ \ \  \text{ on } \mathcal{E}_{\text{bound},n}(\mu).
\]

Moreover one may apply Lemma \ref{lemma_R_n_z_n} to get that 
\begin{equation}\label{QRQ}
   1_{\mathcal{E}_{\text{bound},n}(\mu)}   \left| \|Q_{\ell,n}^*R_n(\lambda_{\ell,n})U_na_n\|^2 -\|Q_{\ell,n}^*R_n(\mu)U_n\tilde{a}_{\mu,n}\|^2 \right| \xrightarrow{\Prob}0.
\end{equation}

Furthermore if one sets $\hat w_n=U_n\tilde{a}_{\mu,n}$ we have that $\hat w_n\in F_{n}$. Indeed one has that $U_n \tilde{a}_{\mu_n}=  U_{\mu_n,n} a_{\mu,n}$ and so $$(E_n U_{\mu_n,n}- \mu U_{\mu_n,n})a_{\mu,n}=0.$$ In particular if one writes $c_n= Q^{*}_{\ell,n} \hat w_n$ then $c_n\in\C^{k_{\ell,n}}$, $\hat w_n=Q_{\ell,n}c_n$ and
\begin{equation}\label{norm_of_c_n}
  \|c_n\|=  \|\hat w_n\| \ \ \text{ on  }\mathcal{E}_{\text{bound},n}(\mu).
\end{equation}
We write
\[
Q_{\ell,n}^*R_n(\mu)\hat w_n = Q_{\ell,n}^*R_n(\mu)Q_{\ell,n}c_n = A_nc_n.
\]
where $A_n=Q_{\ell,n}^*R_n(\mu)Q_{\ell,n}.$

It remains to notice that by Corollary \ref{cor:Mn-converges-outside} 
\begin{equation}\label{A_nc_n}
 1_{\mathcal{E}_{\text{bound},n}(\mu)}  \frac{1}{\|c_n\|^2}\left\| A_nc_n - \frac{c_n}{\mu}  \right\| =1_{\mathcal{E}_{\text{bound},n} 
(\mu)}  \frac{1}{\|c_n\|^2}    \left\| A_nc_n - \frac{1}{\mu} Q_{\ell,n}^*Q_{\ell,n}c_n  \right\| \xrightarrow{\Prob} 0.
\end{equation}
By combining \eqref{A_nc_n}, \eqref{norm_of_c_n} and \eqref{QRQ} we get that 
\begin{equation}\label{numerator}
 1_{\mathcal{E}_{\text{bound},n}(\mu)}   \frac{1}{\|U_n a_n \|^2}\|Q_{\ell,n}^*R_n(\lambda_{\ell,n})U_na_n\|^2 \xrightarrow{\Prob} \frac{1}{|\mu|^2}. 
\end{equation}
Moreover by Corollary \ref{cor_bound_matrix_product_norm} and \eqref{R_n(z_n)-R_n(z)} we get that
\begin{equation}\label{R_N(lambd_n)to1/m^2-1}
1_{\mathcal{E}_{\text{bound},n}(\mu)}  \frac{\|R_n(\lambda_{\ell,n})U_na_n\|^2}{\|U_n a_n\|^2} \xrightarrow{\Prob} \frac{1}{|\mu|^2-1}.
\end{equation}
The proof completes by combining \eqref{numerator} and \eqref{R_N(lambd_n)to1/m^2-1}.
\end{proof}

Next we prove Theorem \ref{main_thm}(b).

\begin{proof}[Proof of Theorem \ref{main_thm}(b)]
We start by using Lemma \ref{lem:ker-evec-bijection} to get that,
\[
 1_{\mathcal{E}_{\text{bound},n}(\mu)}  Q_{\nu,n}^*\tilde u_{\ell,n}
=
 1_{\mathcal{E}_{\text{bound},n}(\mu)}  \frac{Q_{\ell',n}^*R_n(\lambda_{\ell,n})U_na_n}{\|R_n(\lambda_{\ell,n})U_na_n\|}.
\]
As in the proof of Theorem \ref{main_thm}(a) one may conclude that
\[
 1_{\mathcal{E}_{\text{bound},n}(\mu)}  \|Q_{\ell',n}^*\tilde u_{\ell,n}\|^2
= 1_{\mathcal{E}_{\text{bound},n}(\mu)}  
\frac{(|\mu|^2-1)}{\|c_n\|^2}\,\|Q_{\ell',n}^*R_n(\mu)Q_{\ell,n}c_n\|^2 + o_{\Prob}(1),
\]
for some sequence $c_n \in \C^{k_{\ell,n}}$ such that $ \|c_n\|=\|U_na_n\|^2 + o_{\mathbb{P}}(1)$.
It remains to notice that due Lemma \ref{isotropic}
\[
1_{\mathcal{E}_{\text{bound},n}(\mu)}  
\Big|\ \|Q_{\ell',n}^*R_n(\mu)Q_{\ell,n}c_n\|^2-\frac{1}{|\mu|}\,\|Q_{\ell',n}^*Q_{\ell,n}c_n\| \ \Big|
=o_{\Prob}(1).
\]
\end{proof}

\bibliographystyle{alpha}

\bibliography{references}

\newcommand{\etalchar}[1]{$^{#1}$}
\begin{thebibliography}{ABC{\etalchar{+}}24}

\bibitem[ABC{\etalchar{+}}24]{akj-etal-24}
I.~Akjouj, M.~Barbier, M.~Clenet, W.~Hachem, M.~Ma\"ida, F.~Massol, J.~Najim, and V.~C. Tran.
\newblock Complex systems in ecology: a guided tour with large lotka--volterra models and random matrices.
\newblock {\em Proceedings of the Royal Society A}, 480(2285):20230284, 2024.

\bibitem[BBAP05]{baik2005phase}
J.~Baik, G.~Ben~Arous, and S.~P{\'e}ch{\'e}.
\newblock Phase transition of the largest eigenvalue for nonnull complex sample covariance matrices.
\newblock {\em The Annals of Probability}, 33(5):1643--1697, 2005.

\bibitem[BC16]{bordenave2016outlier}
C.~Bordenave and M.~Capitaine.
\newblock Outlier eigenvalues for deformed iid random matrices.
\newblock {\em Communications on Pure and Applied Mathematics}, 69(11):2131--2194, 2016.

\bibitem[BCGZ22]{bordenave2022convergence}
C.~Bordenave, D.~Chafa{\"\i}, and D.~Garc{\'\i}a-Zelada.
\newblock Convergence of the spectral radius of a random matrix through its characteristic polynomial.
\newblock {\em Probability Theory and Related Fields}, pages 1--19, 2022.

\bibitem[BGN11]{benaych2011eigenvalues}
F.~Benaych-Georges and R.~R. Nadakuditi.
\newblock The eigenvalues and eigenvectors of finite, low rank perturbations of large random matrices.
\newblock {\em Advances in Mathematics}, 227(1):494--521, 2011.

\bibitem[BS06]{baik2006eigenvalues}
J.~Baik and J.~W. Silverstein.
\newblock Eigenvalues of large sample covariance matrices of spiked population models.
\newblock {\em Journal of multivariate analysis}, 97(6):1382--1408, 2006.

\bibitem[Bun17]{bunin2017ecological}
G.~Bunin.
\newblock Ecological communities with lotka-volterra dynamics.
\newblock {\em Physical Review E}, 95(4):042414, 2017.

\bibitem[BvH24]{brailovskaya2024universality}
T.~Brailovskaya and R.~van Handel.
\newblock Universality and sharp matrix concentration inequalities.
\newblock {\em Geometric and Functional Analysis}, 34(6):1734--1838, 2024.

\bibitem[CCF09]{capitaine2009largest}
M.~Capitaine, Donati-Martin C., and D.~F{\'e}ral.
\newblock {The largest eigenvalues of finite rank deformation of large Wigner matrices: Convergence and nonuniversality of the fluctuations}.
\newblock {\em The Annals of Probability}, 37(1):1 -- 47, 2009.

\bibitem[CLZ23]{cos-lam-yiz-24}
S.~Coste, G.~Lambert, and Y.~Zhu.
\newblock The characteristic polynomial of sums of random permutations and regular digraphs.
\newblock {\em International Mathematics Research Notices}, 2024(3):2461--2510, 2023.

\bibitem[Cos23]{coste2023sparse}
S.~Coste.
\newblock Sparse matrices: convergence of the characteristic polynomial seen from infinity.
\newblock {\em Electronic Journal of Probability}, 28:1--40, 2023.

\bibitem[HL26]{hachem2026spectral}
Walid Hachem and Michail Louvaris.
\newblock On the spectral radius and the characteristic polynomial of a random matrix with independent elements and a variance profile.
\newblock {\em The Annals of Applied Probability}, 2026.

\bibitem[HLN26]{hachem2026extreme}
Walid Hachem, Michail Louvaris, and Jamal Najim.
\newblock Extreme eigenvalues and eigenvectors for finite rank additive deformations of non-hermitian sparse random matrices.
\newblock {\em arXiv preprint arXiv:2602.20956}, 2026.

\bibitem[Pau07]{paul2007asymptotics}
D.~Paul.
\newblock Asymptotics of sample eigenstructure for a large dimensional spiked covariance model.
\newblock {\em Statistica Sinica}, pages 1617--1642, 2007.

\bibitem[SCS88]{som-cri-som-88}
H.~Sompolinsky, A.~Crisanti, and H.~J. Sommers.
\newblock Chaos in random neural networks.
\newblock {\em Phys. Rev. Lett.}, 61:259--262, Jul 1988.

\bibitem[Tao13]{tao2013outliers}
T.~Tao.
\newblock Outliers in the spectrum of iid matrices with bounded rank perturbations.
\newblock {\em Probability Theory and Related Fields}, 155(1):231--263, 2013.

\bibitem[WT13]{wai-tou-13}
G.~Wainrib and J.~Touboul.
\newblock Topological and dynamical complexity of random neural networks.
\newblock {\em Phys. Rev. Lett.}, 110:118101, Mar 2013.

\end{thebibliography}
\addcontentsline{toc}{chapter}{bibliography}

\end{document}